\documentclass{article}

\usepackage{amsmath,amssymb}
\usepackage[dvipdfmx]{graphicx}
\usepackage{wrapfig}
\usepackage{here}
\usepackage{amsthm}
\usepackage{cancel}
\usepackage{mathtools}
\usepackage[all]{xy}

\newcommand{\Ker}{\mathop{\mathrm{Ker}}}
\newcommand{\im}{\mathop{\mathrm{Im}}}

\usepackage{tikz}
\usetikzlibrary{shapes.geometric}
\usetikzlibrary {shapes.misc}
\usetikzlibrary{intersections, calc, arrows, positioning, arrows.meta}

\newtheorem{dfn}{Definition}
\newtheorem{lem}{Lemma}
\newtheorem{thm}{Theorem}

\newtheorem{exm}{Example}
\newtheorem{rmk}{Remark}

\usepackage{color}

\ifx \ocirc \undefined \def \ocirc #1{{\accent'27#1}}\fi

\title{Why it is sufficient to consider only the case where the seed of linear cellular automata is $1$}
\author{Akane Kawaharada\thanks{
Address: 1, Fujinomoricho, Fukakusa, Fushimi-ku, Kyoto, 612-8522, JAPAN.  
E-mail: aka@kyokyo-u.ac.jp}
\\\vspace{2pt}
Department of Mathematics, Kyoto University of Education}
\date{\today}

\begin{document}
\maketitle
\begin{abstract}
When using a cellular automaton (CA) as a fractal generator, consider orbits from the single site seed, an initial configuration that gives only a single cell a positive value.
In the case of a two-state CA, since the possible states of each cell are $0$ or $1$, the ``seed" in the single site seed is uniquely determined to be the state $1$.
However, for a CA with three or more states, there are multiple candidates for the seed.
For example, for a $3$-state CA, the possible states of each cell are $0$, $1$, and $2$, so the candidates for the seed are $1$ and $2$.
For a $4$-state CA, the possible states of each cell are $0$, $1$, $2$, and $3$, so the candidates for the seed are $1$, $2$, and $3$.
Thus, as the number of possible states of a CA increases, the number of seed candidates also increases.
In this paper, we prove that for linear CAs it is sufficient to consider only the orbit from the single site seed with the seed $1$.
% (176 words)
\end{abstract}

\hspace{2.5mm} 

%%%%%%%%%%%%%%%%%%%%%%%%%%%%%%%%%%%%%
%%%%%%%%%%%%%%%%%%%%%%%%%%%%%%%%%%%%%

\section{Introduction}
\label{sec:int}

Research on cellular automata (CAs) is diverse.
Research is being conducted on a wide range of topics, including analysis as topological dynamical systems \cite{hedlund1969, blanchard1997, kurka2001}, research as fractal generators \cite{amoroso1971, ostrand1971, willson1984}, and research as mathematical models for biological phenomena, chemical phenomena, social phenomena, etc. \cite{coombes2009, gerhardt1989, fuks1997, matsumoto1998}.

\begin{figure}[htbp]
\centering
\includegraphics[width=.45\linewidth]{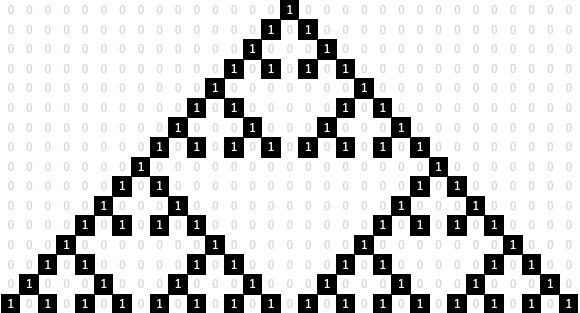}
\caption{Spatio-temporal pattern of a $2$-state linear CA from the single site seed.}
% \caption{Spatio-temporal pattern $\{T^t u_{\langle a \rangle}\}_{t=0}^{15}$ from the initial configuration $u_{\langle a \rangle}$ of a $2$-state linear CA given by $(T u)_i = u_{i-1} + u_{i+1}$ on $(\mathbb Z/ 2 \mathbb Z)^{\mathbb Z}$.}
\label{fig:sp1}
\end{figure}

In studies using CAs as fractal generators, orbits from initial configurations with a finite number of cells given positive values have generally been studied.
% \cite{amoroso1971, ostrand1971, willson1984}.
The simplest initial configuration is the single site seed, where only a single cell is given a positive value, and the studies of the orbits from the single site seeds are fundamental but important \cite{takahashi1992, haeseler1993, wolfram2002}.
In the case of a $2$-state CA, since the possible states of each cell are $0$ and $1$, the seed in the single site seed is uniquely determined to be $1$ (e.g. Figure~\ref{fig:sp1}).
However, for a CA with three or more states, there are multiple candidates for the state that the seed can take.
For example, for a $3$-state CA, the possible states of each cell are $0$, $1$, and $2$, so there are two candidates, $1$ or $2$, for the seed. %(see Figure~\ref{fig:sp3}).
For a $4$-state CA, the possible states of each cell are $0$, $1$, $2$, and $3$, so the candidates of the seed are $1$, $2$, or $3$. %(see Figure~\ref{fig:sp4}).
Thus, there are many candidates for the seed in a multi-state CA, but in previous studies, the seed is often considered only in the case of $1$.

In this study, we prove that it is sufficient to consider the orbit only when the seed of the single site seed is $1$ for linear CAs.
The proof will be divided into the following three cases.
First, when the number of states is prime, we show that the spatio-temporal patterns of a linear CA from the single site seed are all isomorphic regardless of seed.
Next, when the number of states is composite, the proof will be divided into the following two cases.
If the number of states is composite and the seed is relatively prime to the number of states, we show that their spatio-temporal patterns are isomorphic to each other even if the seed is switched.
If the number of states is composite and the seed is not relatively prime to the number of states, then the spatio-temporal pattern is isomorphic to a spatio-temporal pattern of a CA with the smaller number of states.
Integrating these three cases, we can say that any spatio-temporal pattern of any seed with any number of states is isomorphic to a spatio-temporal pattern of the seed $1$ with some number of states.

The remainder of this paper is organized as follows.
In Section~\ref{sec:int}, the background and motivation for this study are described.
Section~\ref{sec:pre} gives definitions and notations on linear CAs and the group theory that is necessary for the proof of the main result.
Section~\ref{sec:main} gives the main result. 
We divide the three cases according to the number of states and the seed of linear CAs, prove each case, and derive the main theorem.
In Section~\ref{sec:con}, we give a summary and discuss future works.

\section{Preliminaries}
\label{sec:pre}

In this paper, we analyze the asymptotic behavior of orbits of linear CAs based on group theory. 
Section~\ref{subsec:def_group} gives notations and definitions in the group theory, and Section~\ref{subsec:def_ca} gives the notations and definitions of CAs.

\subsection{Notations and definitions of groups}
\label{subsec:def_group}

Let $n \in {\mathbb Z}_{\geq 2}$.
Let ${\mathbb Z} / n {\mathbb Z} = \{0, 1, \ldots, n-1\}$ be the set of the residue classes of integers modulo $n$, and it forms an additive group, because $(i)$ there exists the identity element $0$, $(ii)$ for each $a \in {\mathbb Z} / n{\mathbb Z}$, there exists the inverse element denoted by $n-a \pmod n$, and $(iii)$ the associative law is satisfied.
The additive group ${\mathbb Z} / n {\mathbb Z}$ is a cyclic group of order $n$ under addition.
Since a cyclic group is an abelian group, the additive tables of the group ${\mathbb Z} / n {\mathbb Z}$ (for example, see Tables~\ref{tab:mod3} and \ref{tab:mod5}) are line symmetric with respect to the diagonal.
It is known that any subgroup of ${\mathbb Z} / n {\mathbb Z}$ is also a cyclic group, and for each divisor $r$ of finite $n$, there exists just one subgroup of order $r$.

If the representative element of the residue class is relatively prime to $n \in {\mathbb Z}_{\geq 2}$, it is called the irreducible residue class modulo $n$, and the set of them is written as $({\mathbb Z} / n {\mathbb Z})^{\times}$.
Since it forms a group under multiplication, we call it the irreducible residue class group modulo $n$.
Because the generators of the additive group ${\mathbb Z} / n {\mathbb Z}$ are elements relatively prime to $n$, the elements in $({\mathbb Z} / n {\mathbb Z})^{\times}$ are the generators of the additive group ${\mathbb Z} / n {\mathbb Z}$.

\subsection{Notations and definitions of CAs}
\label{subsec:def_ca}

Here, we consider linear CAs. 
Let the dimension $D$ of a CA be any finite positive integer, and the number of states $n$ be any finite integer greater than or equal to $2$.

\begin{dfn}
\label{dfn:ca}
Let $m$ be a finite positive integer, $c_j$ be a coefficient in ${\mathbb Z}$, and ${\boldsymbol v}_j$ be a vector on ${{\mathbb Z}^D}$ for $1 \leq j \leq m$.
A transition rule $T$ of a $D$-dimensional $n$-state linear CA $(({\mathbb Z} / n {\mathbb Z})^{{\mathbb Z}^D}, T)$ is given by
\begin{align}
(T u)_{\boldsymbol i} = \sum_{j=1}^m c_j u_{{\boldsymbol i}+ {\boldsymbol v}_j} \quad %\pmod n
\end{align} 
for $u = \{u_{\boldsymbol i}\}_{{\boldsymbol i} \in {{\mathbb Z}^D}} \in ({\mathbb Z} / n {\mathbb Z})^{{\mathbb Z}^D}$.
\end{dfn}

\begin{rmk}
\label{rmk:gr}
The transition rule $T$ of a linear CA $(({\mathbb Z} / n {\mathbb Z})^{{\mathbb Z}^D}, T)$ can be regarded as a finite number of iterations of additive operations for the state $u_{\boldsymbol i}$ on the additive group ${\mathbb Z} / n {\mathbb Z}$.
More specifically, it can be regarded as repeating $\sum_{j=1}^m c_j$ times additive operations for states at neighboring cells ${\boldsymbol i}+ {\boldsymbol v}_j$ for $1 \leq j \leq m$.
Thus, the discussion on a transition rule of a linear CA can be replaced by the discussion on operations on the additive group ${\mathbb Z} / n {\mathbb Z}$.
\end{rmk}

Suppose that $T^t$ is a $t$ times composition of the transformation $T$.
In particular, $T^0 u=u$.
We give a configuration $u_{\langle a \rangle} \in ({\mathbb Z} / n {\mathbb Z})^{{\mathbb Z}^D}$ for $a \in ({\mathbb Z} / n {\mathbb Z}) \backslash \{0\}$ by
\begin{align}
( u_{\langle a \rangle} )_{\boldsymbol i} =
\left\{
\begin{array}{l l}
a & \mbox{if } {\boldsymbol i} = {\boldsymbol 0} \coloneqq \{0, 0, \ldots, 0\} \in {{\mathbb Z}^D},\\
0 & \mbox{otherwise}.
\end{array}
\right. 
\end{align}
The state $a$ with ${\boldsymbol i} = {\boldsymbol 0}$ is called a seed and we call the configuration $u_{\langle a \rangle}$ the single site seed with $a$.
In this study, we consider the orbits from the single site seed with $a$ of linear CAs.

\begin{dfn}
\label{dfn:iso}
For $n, \hat{n} \in {\mathbb Z}_{\geq 2}$, let $(({\mathbb Z} / n {\mathbb Z})^{{\mathbb Z}^D}, T)$ be a $D$-dimensional $n$-state linear CA, and let $(({\mathbb Z} / \hat{n} {\mathbb Z})^{{\mathbb Z}^D}, T)$ be a $D$-dimensional $\hat{n}$-state linear CA.
Note that these transition rules $T$ are the same.
We consider the two spatio-temporal patterns of the CAs from the single site seed with $a \in ({\mathbb Z} / n {\mathbb Z}) \backslash \{0\}$ and $\hat{a} \in ({\mathbb Z} / \hat{n} {\mathbb Z}) \backslash \{0\}$, denoted by $\{T^t u_{\langle a \rangle}\}_{t=0}^{\infty}$ and $\{T^t u_{\langle \hat{a} \rangle}\}_{t=0}^{\infty}$, respectively.
Let $\Delta_T(n, a)$ be a subset of ${\mathbb Z} / n {\mathbb Z}$ given by 
\begin{align}
\Delta_T(n, a) := \{ b \in {\mathbb Z} / n {\mathbb Z} \mid \mbox{There exists $({\boldsymbol i}, t) \in ({\mathbb Z}^D, {\mathbb Z}_{\geq 0})$ such that $(T^t u_{\langle a \rangle})_{\boldsymbol i} = b$} \}.
\end{align}
If there exists an isomorphism $f : \Delta_T(n, a) \to \Delta_T(\hat{n}, \hat{a})$ such that $f (T^t u_{\langle a \rangle})_{\boldsymbol i} = ( T^t u_{\langle \hat{a} \rangle})_{\boldsymbol i}$ for any site ${\boldsymbol i} \in {\mathbb Z}^D$ and any time step $t \in {\mathbb Z}_{\geq 0}$, then we call $\{T^t u_{\langle a \rangle}\}_{t=0}^{\infty}$ and $\{T^t u_{\langle \hat{a} \rangle}\}_{t=0}^{\infty}$ isomorphic, and we denote their relationship by 
\begin{align}
S_T(n, a) \cong S_T (\hat{n}, \hat{a}).
% \{T^t u_{\langle a \rangle}\}_{t=0}^{\infty} \cong \{\hat{T}^t u_{\langle \hat{a} \rangle}\}_{t=0}^{\infty}.
\end{align}
\end{dfn}

%%%%%%%%%%%%%%%%%%%%%%%%%%%%%%%%%%%%%%%%%%

\section{Main results}
\label{sec:main}

In this section, we prove that it is sufficient to consider only when the seed is $1$ for the spatio-temporal patterns of linear CAs.

As shown in Figure~\ref{fig:flow}, the proof will be divided into three cases classified by whether the number of states $n$ of a linear CA is prime or not, and if not prime whether the seed $a$ is relatively prime to $n$ or not.
This section consists of the following four parts.
In Section~\ref{subsec:pri}, we show that for the case when the number of states $n$ of a linear CA is prime in Lemma~\ref{lem:main_pri}, their spatio-temporal patterns are isomorphic, no matter what the seeds are.
In Section~\ref{subsec:com1}, we show that spatio-temporal patterns are isomorphic when the number of states $n$ of a linear CA is composite and the seed $a$ is relatively prime to $n$ in Lemma~\ref{lem:main_com1} .
In section~\ref{subsec:com2}, we show that in Lemma~\ref{lem:main_com2}, when the number of states $n$ of a linear CA is composite and the seed $a$ is not relatively prime to $n$, the spatio-temporal pattern is isomorphic to the spatio-temporal pattern of a linear CA with the same transition rule when the number of state is $n/\gcd(n, a)$ and the seed is $a/\gcd(n, a)$, where $\gcd(n, a)$ denotes the greatest common divisor of $n$ and $a$.
Section~\ref{subsec:all} summarizes the results of the above three cases to obtain Theorem~\ref{thm:main}.
Theorem~\ref{thm:main} shows that for any integer $n$ greater than or equal to $2$ and any positive integer $a$ less than $n$, the spatio-temporal pattern of a $n$-state linear CA from the initial configuration with the seed $a$ is isomorphic to the spatio-temporal pattern of a $(n/\gcd(n, a))$-state linear CA from the initial configuration with the seed $1$.

\begin{figure}[htb]
\centering
\includegraphics[width=1.\linewidth]{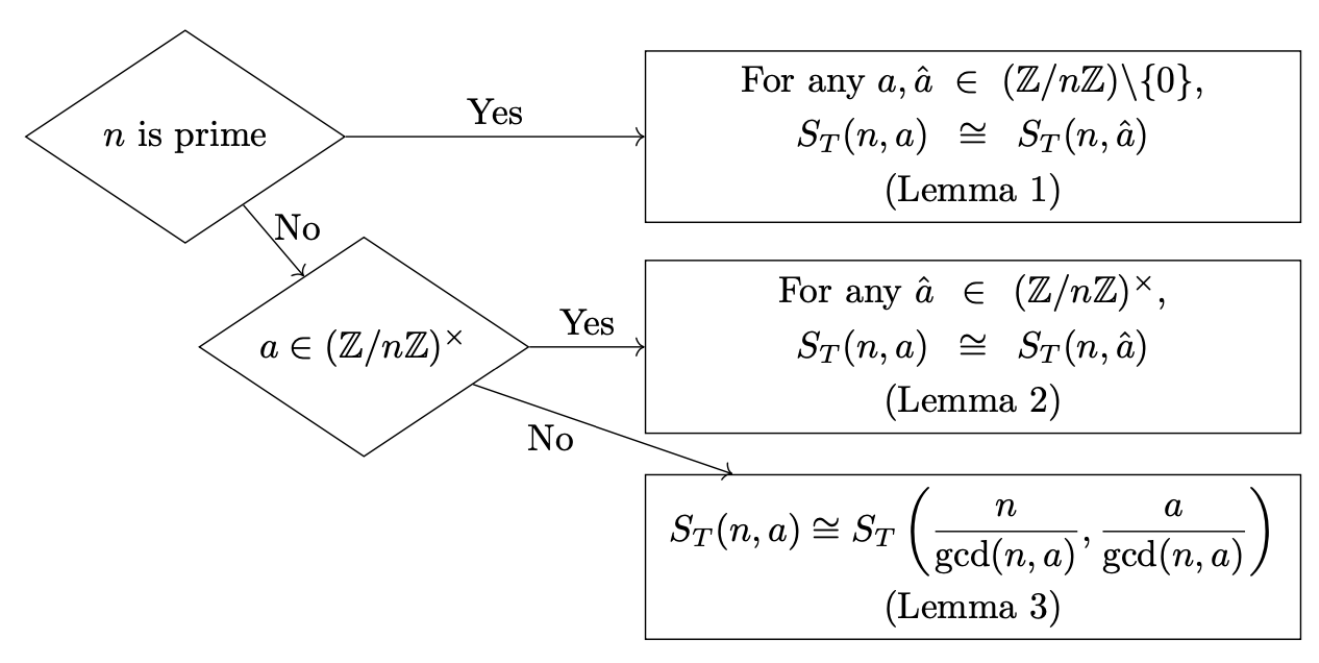}
\caption{Branching of proofs by the number of states $n$ and the seed $a$}
\label{fig:flow}
\end{figure}

\subsection{When the number of states is prime}
\label{subsec:pri}

Here, we consider the spatio-temporal patterns of linear CAs when the number of states is prime.
An example shown in Figure~\ref{fig:sp3} is an expansion of the one-dimensional elementary CA Rule $90$ \cite{wolfram2002} to the case of $3$ number of states. 
The transition rule is given by $(T u)_i = u_{i-1} + u_{i+1}$ on $(\mathbb Z/ 3 \mathbb Z)^{\mathbb Z}$, where the possible states of each cell are $0$, $1$ or $2$.
Figure~\ref{fig:sp3}~$(a)$ shows the spatio-temporal pattern when the seed $a$ is $1$, and Figure~\ref{fig:sp3}~$(b)$ shows the spatio-temporal pattern when the seed is $2$.
In these patterns, not only the sites of cells that take positive states match, but it can be seen that the states $1$ and $2$ are completely switched.
Figure~\ref{fig:sp5} shows the spatio-temporal patterns when the number of states is $5$ with the same transition rule as in Figure~\ref{fig:sp3}.
In this case, the possible states of each cell are $0$, $1$, $2$, $3$, or $4$.
Figure~\ref{fig:sp5}~$(a)$ through \ref{fig:sp5}~$(d)$ show the patterns when the seed $a$ is $1$, $2$, $3$, and $4$, respectively.
We can see that the sites of positive-valued cells match in the four patterns, and positive states are switched.
We will see the details of which states are swapped in Example~\ref{exm:lem1_5}, but for example, the states $1$, $2$, $3$, and $4$ in Figure~\ref{fig:sp5}~$(a)$ at seeded $1$ are swapped to $2$, $4$, $1$, and $3$  in Figure~\ref{fig:sp5}~$(b)$ at seeded $2$, respectively. 
Thus, we will show in Lemma~\ref{lem:main_pri} that the patterns consisting of positive-valued cells of a linear CA with prime number of states are invariant when the seed is changed, and that in the spatio-temporal pattern the positive-valued states are swapped.

\begin{figure}[htbp]
\begin{minipage}[b]{0.45\linewidth}
\centering
\includegraphics[width=1.\linewidth]{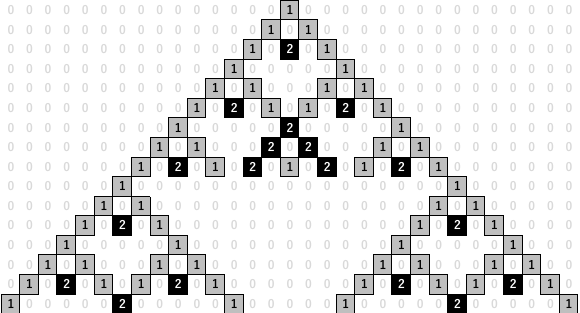}\\
$(a)$ when $a=1$
\end{minipage}
\begin{minipage}[b]{0.05\linewidth}
\quad
\end{minipage}
\begin{minipage}[b]{0.45\linewidth}
\centering
\includegraphics[width=1.\linewidth]{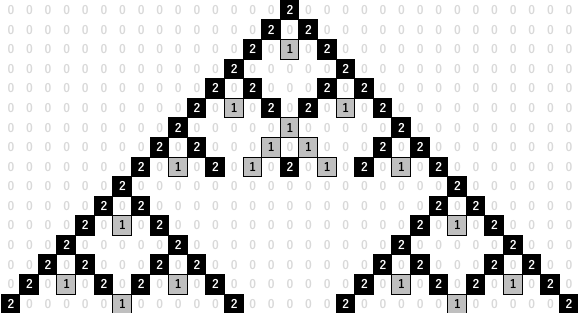}\\
$(b)$ when $a=2$
\end{minipage}
\caption{Spatio-temporal patterns $\{T^t u_{\langle a \rangle}\}_{t=0}^{15}$ from the initial configuration $u_{\langle a \rangle}$ of a $3$-state linear CA given by $(T u)_i = u_{i-1} + u_{i+1}$ on $(\mathbb Z/ 3 \mathbb Z)^{\mathbb Z}$.}
\label{fig:sp3}
\end{figure}

\begin{figure}[htbp]
\begin{minipage}[b]{0.45\linewidth}
\centering
\includegraphics[width=1.\linewidth]{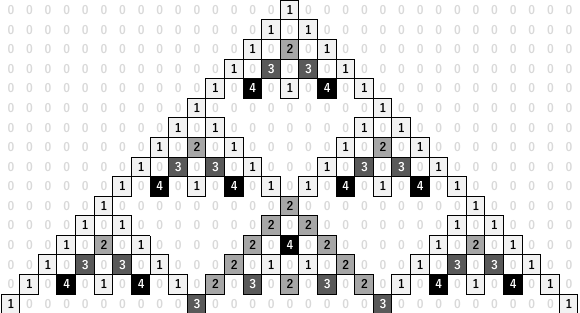}\\
$(a)$ when $a=1$
\end{minipage}
\begin{minipage}[b]{0.05\linewidth}
\quad
\end{minipage}
\begin{minipage}[b]{0.45\linewidth}
\centering
\includegraphics[width=1.\linewidth]{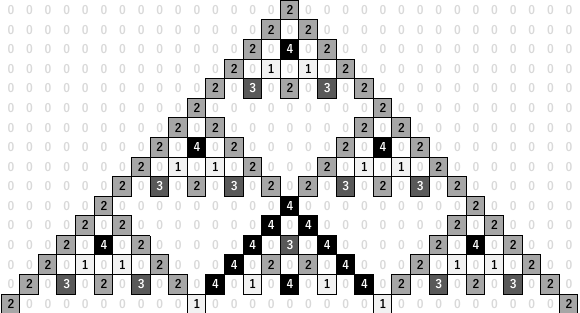}\\
$(b)$ when $a=2$
\end{minipage}\\
\quad \\
\begin{minipage}[b]{0.45\linewidth}
\centering
\includegraphics[width=1.\linewidth]{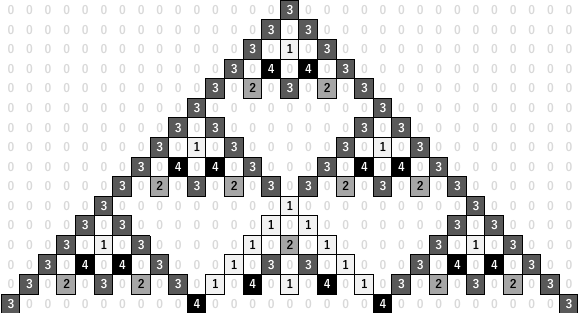}\\
$(c)$ when $a=3$
\end{minipage}
\begin{minipage}[b]{0.05\linewidth}
\quad
\end{minipage}
\begin{minipage}[b]{0.45\linewidth}
\centering
\includegraphics[width=1.\linewidth]{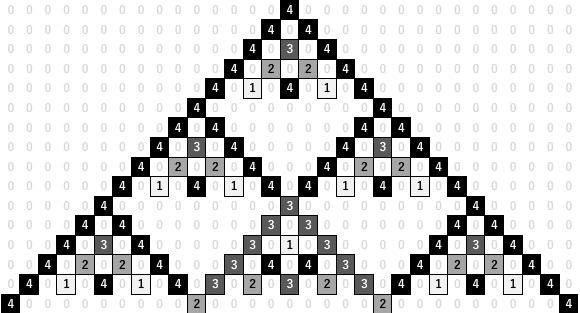}\\
$(d)$ when $a=4$
\end{minipage}
\caption{Spatio-temporal patterns $\{T^t u_{\langle a \rangle}\}_{t=0}^{15}$ from the initial configuration $u_{\langle a \rangle}$ of a $5$-state linear CA given by $(T u)_i = u_{i-1} + u_{i+1}$ on $(\mathbb Z/ 5 \mathbb Z)^{\mathbb Z}$.}
\label{fig:sp5}
\end{figure}

\begin{lem}
\label{lem:main_pri}
Let $p$ be a prime number.
For a $D$-dimensional $p$-state linear CA $(({\mathbb Z} / p {\mathbb Z})^{{\mathbb Z}^D}, T)$, 
\begin{align}
S_T(p, a) \cong S_T(p, \hat{a})
% \{T^t u_{\langle a \rangle}\}_{t=0}^{\infty} \cong \{T^t u_{\langle \hat{a} \rangle}\}_{t=0}^{\infty}
\end{align}
for any $a, \hat{a} \in ({\mathbb Z} / p {\mathbb Z}) \backslash \{0\}$.
\end{lem}

\begin{proof}
Recall that ${\mathbb Z} / p {\mathbb Z}$ is a cyclic group. 
Since $p$ is a prime number, all elements of ${\mathbb Z} / p {\mathbb Z}$ except the identity element $0$ are generators of the group, and their orders are $p$.
This fact guarantees that when the seed $a$ is any state other than $0$, all states in ${\mathbb Z} / p {\mathbb Z}$ can appear in the spatio-temporal pattern $\{T^t u_{\langle a \rangle}\}_{t=0}^{\infty}$. 
(Note that we are not claiming that all states must appear.)

Let $k \in ({\mathbb Z} / p {\mathbb Z}) \backslash \{0\}$.
We define a map $f_k : {\mathbb Z} / p {\mathbb Z} \to {\mathbb Z} / p {\mathbb Z}$ by $f_k (b) = k b$ for $b \in {\mathbb Z} / p {\mathbb Z}$. 
% Since $p$ is a prime number, and then $({\mathbb Z} / p {\mathbb Z})^{\times} = ({\mathbb Z} / p {\mathbb Z}) \backslash \{0\}$.
When $k=1$, it is just the identity, $f_k (b) = b$.
Because $f_k (b+\hat{b}) = k (b+\hat{b}) = kb + k\hat{b} = f_k(b) + f_k(\hat{b})$ for $b, \hat{b} \in {\mathbb Z} / p {\mathbb Z}$, $f_k$ is a homomorphism on ${\mathbb Z} / p {\mathbb Z}$.
Because ${\mathbb Z} / p {\mathbb Z}$ is a simple group, it is easy to see that $\Ker f_k = \{0\}$ and $\im f_k = {\mathbb Z} / p {\mathbb Z}$, and then $f_k$ is a group automorphism.
% Hence, for each element of ${\mathbb Z} / p {\mathbb Z}$ arranged in generation order from generator $1$, $f_k$ gives the corresponding element arranged in generation order from generator $a$.

Therefore, for $\hat{a} \in ({\mathbb Z} / p {\mathbb Z}) \backslash \{0\}$, we can take $a \in ({\mathbb Z} / p {\mathbb Z}) \backslash \{0\}$ such that $(T^t u_{\langle a \rangle})_{\boldsymbol i} = f_k (T^t u_{\langle \hat{a} \rangle})_{\boldsymbol i}$ for any site ${\boldsymbol i} \in {\mathbb Z}^D$ and any time step $t \in {\mathbb Z}_{\geq 0}$.
\end{proof}

We will see specifically at the results shown by Lemma~\ref{lem:main_pri} for the cases of $3$ and $5$ number of states.

\begin{exm}[When $n=3$]
\label{exm:lem1_3}

Table~\ref{tab:f_k3}~$(a)$ shows the generation orders of the additive group ${\mathbb Z} / 3 {\mathbb Z}$ by the generators, $1$ and $2$.
From generator $1$, we have a generation order $1 \mapsto 2 \mapsto 0$, and from generator $2$, we have a generation order $2 \mapsto 1 \mapsto 0$. 
Table~\ref{tab:f_k3}~$(b)$ shows the values of $f_a$ on ${\mathbb Z} / 3 {\mathbb Z}$.
Given $b$ in generation order $1 \mapsto 2 \mapsto 0$, the values of $f_a(b)$ for them form the generation order of ${\mathbb Z} / 3 {\mathbb Z}$ for $a$ as a generator.
When $a=1$, $(f_1 (1), f_1 (2), f_1 (0)) = (1, 2, 0)$ and when $a=2$, $(f_2 (1), f_2 (2), f_2 (0)) = (2, 1, 0)$.
Hence, we can confirm that these orders by $f_a$ correspond to the generation order by each generator $a$.

\begin{table}[htb]
\caption{Generation orders of the group ${\mathbb Z} / 3 {\mathbb Z}$ by generator $a \in ({\mathbb Z} / 3 {\mathbb Z})^{\times}$ and values of $f_a$ on ${\mathbb Z} / 3 {\mathbb Z}$}
\label{tab:f_k3}
\quad \\
\begin{minipage}[c]{0.45\hsize}
\centering
$(a)$ Generation orders of ${\mathbb Z} / 3 {\mathbb Z}$ by $a$\\
\vspace{2mm}
\begin{tabular}{c l}
\hline
$a$ & generation order\\
\hline
$1$ & $1$ $\mapsto$ $2$ $\mapsto$ $0$\\
$2$ & $2$ $\mapsto$ $1$ $\mapsto$ $0$\\
\hline
\end{tabular}
\end{minipage}
\begin{minipage}[c]{0.05\hsize}
\quad
\end{minipage}
\begin{minipage}[c]{0.45\hsize}
\centering
$(b)$ Values of $f_a$ for $b \in {\mathbb Z} / 3 {\mathbb Z}$\\
\vspace{2mm}
\begin{tabular}{c | c c c}
\hline
$b$ & $1$ & $2$ & $0$\\
% \hline
$f_1(b)$ & $1$ & $2$ & $0$\\
$f_2(b)$ & $2$ & $1$ & $0$\\
\hline
\end{tabular}
\end{minipage}
\end{table}

Next, we consider the additive tables of ${\mathbb Z} / 3 {\mathbb Z}$ where elements are ordered by the generation order.
Table~\ref{tab:mod3}~$(a)$ is the additive table arranged in the order of generation from generator $1$, and Table~\ref{tab:mod3}~$(b)$ is the additive table arranged in the order of generation from generator $2$.
Because $f_a$ for $a \in ({\mathbb Z} / 3 {\mathbb Z})^{\times}$ is a group automorphism, comparing these two tables, we can notice that the identity element $0$ appears the same places, while the other elements $1$ and $2$ are switched.

Therefore, Lemma~\ref{lem:main_pri} shows that the states $1$ and $2$ in the spatio-temporal pattern when the seed is $1$ are replaced by the states $2$ and $1$ in the spatio-temporal pattern when the seed is $2$, respectively.

\begin{table}[htb]
\caption{Additive tables of ${\mathbb Z} / 3 {\mathbb Z}$}
\label{tab:mod3}
\quad \\
\begin{minipage}[c]{0.45\hsize}
\centering
$(a)$ Table arranged in the order of generation from generator $1$\\
\vspace{2mm}
\begin{tabular}{c| c c c}
$+$ & $1$ & $2$ & $0$\\
\hline
$1$ & $2$ & $0$ & $1$\\
$2$ & $0$ & $1$ & $2$\\
$0$ & $1$ & $2$ & $0$\\
\end{tabular}
\end{minipage}
\begin{minipage}[c]{0.05\hsize}
\quad
\end{minipage}
\begin{minipage}[c]{0.45\hsize}
\centering
$(b)$ Table arranged in the order of generation from generator $2$\\
\vspace{2mm}
\begin{tabular}{c| c c c}
$+$ & $2$ & $1$ & $0$\\
\hline
$2$ & $1$ & $0$ & $2$\\
$1$ & $0$ & $2$ & $1$\\
$0$ & $2$ & $1$ & $0$\\
\end{tabular}
\end{minipage}
\end{table}
\end{exm}

\begin{exm}[When $n=5$]
\label{exm:lem1_5}

Table~\ref{tab:f_k5}~$(a)$ shows the generation orders of the additive group ${\mathbb Z} / 5 {\mathbb Z}$ by generator $a \in ({\mathbb Z} / 5 {\mathbb Z})^{\times} = \{1, 2, 3, 4\}$.
Values of $f_a$ on ${\mathbb Z} / 5 {\mathbb Z}$ are given in Table~\ref{tab:f_k5}~$(b)$, and the order of $b$ in the first row is given by the generation order from the generator $1$.
Comparing these two tables, the generation order from $a$ completely matches the sequence $(f_a(1), f_a(2), f_a(3), f_a(4), f_a(0))$. 
Hence, we can see that the generation order from generator is given by $f_a$ for each element of the generation order from $1$.

\begin{table}[htb]
\caption{Generation orders of the group ${\mathbb Z} / 5 {\mathbb Z}$ by generator $a \in ({\mathbb Z} / 5 {\mathbb Z})^{\times}$ and values of $f_a$ on ${\mathbb Z} / 5 {\mathbb Z}$}
\label{tab:f_k5}
\quad \\
\begin{minipage}[c]{0.45\hsize}
\centering
$(a)$ Generation orders of ${\mathbb Z} / 5 {\mathbb Z}$ by $a$\\
\vspace{2mm}
\begin{tabular}{c l}
\hline
$a$ & generation order\\
\hline
$1$ & $1$ $\mapsto$ $2$ $\mapsto$ $3$ $\mapsto$ $4$ $\mapsto$ $0$\\
$2$ & $2$ $\mapsto$ $4$ $\mapsto$ $1$ $\mapsto$ $3$ $\mapsto$ $0$\\
$3$ & $3$ $\mapsto$ $1$ $\mapsto$ $4$ $\mapsto$ $2$ $\mapsto$ $0$\\
$4$ & $4$ $\mapsto$ $3$ $\mapsto$ $2$ $\mapsto$ $1$ $\mapsto$ $0$\\
\hline
\end{tabular}
\end{minipage}
\begin{minipage}[c]{0.05\hsize}
\quad
\end{minipage}
\begin{minipage}[c]{0.45\hsize}
\centering
$(b)$ Values of $f_a$ for $b \in {\mathbb Z} / 5 {\mathbb Z}$\\
\vspace{2mm}
\begin{tabular}{c| c c c c c}
\hline
$b$ & $1$ & $2$ & $3$ & $4$ & $0$\\
% \hline
$f_1(b)$ & $1$ & $2$ & $3$ & $4$ & $0$\\
$f_2(b)$ & $2$ & $4$ & $1$ & $3$ & $0$\\
$f_3(b)$ & $3$ & $1$ & $4$ & $2$ & $0$\\
$f_4(b)$ & $4$ & $3$ & $2$ & $1$ & $0$\\
\hline
\end{tabular}
\end{minipage}
\end{table}

Table~\ref{tab:mod5} shows the additive tables of ${\mathbb Z} / 5 {\mathbb Z}$ where elements are ordered by the generation order.
Table~\ref{tab:mod5}~$(a)$, $(b)$, $(c)$, and $(d)$ correspond to the cases where the generators are $1$, $2$, $3$, and $4$, respectively.
In these four tables, the identity element $0$ appears at the same sites, and positive states are replaced by other positive states, because $f_a$ for $a \in ({\mathbb Z} / 5 {\mathbb Z})^{\times}$ is a group automorphism.
The states $(1, 2, 3, 4)$ in Table~\ref{tab:mod5}~$(a)$ are replaced by $(2, 4, 1, 3)$ in Table~\ref{tab:mod5}~$(b)$, $(3, 1, 4, 2)$ in Table~\ref{tab:mod5}~$(c)$ and $(4, 3, 2, 1)$ in Table~\ref{tab:mod5}~$(d)$.

Thus, for the spatio-temporal patterns of $5$-state linear CAs, the sites of the positive-valued cells are invariant when the seed of the initial configuration is replaced, indicating that the states are swapped by $f_a$.

\begin{table}[htb]
\caption{Additive tables of ${\mathbb Z} / 5 {\mathbb Z}$}
\label{tab:mod5}
\quad \\
\begin{minipage}[c]{0.45\hsize}
\centering
$(a)$ Table arranged in the order of generation from generator $1$\\
\vspace{2mm}
\begin{tabular}{c| c c c c c}
$+$ & $1$ & $2$ & $3$ & $4$ & $0$\\
\hline
$1$ & $2$ & $3$ & $4$ & $0$ & $1$\\
$2$ & $3$ & $4$ & $0$ & $1$ & $2$\\
$3$ & $4$ & $0$ & $1$ & $2$ & $3$\\
$4$ & $0$ & $1$ & $2$ & $3$ & $4$\\
$0$ & $1$ & $2$ & $3$ & $4$ & $0$\\
\end{tabular}
\end{minipage}
\begin{minipage}[c]{0.05\hsize}
\quad
\end{minipage}
\begin{minipage}[c]{0.45\hsize}
\centering
$(b)$ Table arranged in the order of generation from generator $2$\\
\vspace{2mm}
\begin{tabular}{c| c c c c c}
$+$ & $2$ & $4$ & $1$ & $3$ & $0$\\
\hline
$2$ & $4$ & $1$ & $3$ & $0$ & $2$\\
$4$ & $1$ & $3$ & $0$ & $2$ & $4$\\
$1$ & $3$ & $0$ & $2$ & $4$ & $1$\\
$3$ & $0$ & $2$ & $4$ & $1$ & $3$\\
$0$ & $2$ & $4$ & $1$ & $3$ & $0$\\
\end{tabular}
\end{minipage}
\\ \quad \\ \quad \\
\begin{minipage}[c]{0.45\hsize}
\centering
$(c)$ Table arranged in the order of generation from generator $3$\\
\vspace{2mm}
\begin{tabular}{c| c c c c c}
$+$ & $3$ & $1$ & $4$ & $2$ & $0$\\
\hline
$3$ & $1$ & $4$ & $2$ & $0$ & $3$\\
$1$ & $4$ & $2$ & $0$ & $3$ & $1$\\
$4$ & $2$ & $0$ & $3$ & $1$ & $4$\\
$2$ & $0$ & $3$ & $1$ & $4$ & $2$\\
$0$ & $3$ & $1$ & $4$ & $2$ & $0$\\ 
\end{tabular}
\end{minipage}
\begin{minipage}[c]{0.05\hsize}
\quad
\end{minipage}
\begin{minipage}[c]{0.45\hsize}
\centering
$(d)$ Table arranged in the order of generation from generator $4$\\
\vspace{2mm}
\begin{tabular}{c| c c c c c}
$+$ & $4$ & $3$ & $2$ & $1$ & $0$\\
\hline
$4$ & $3$ & $2$ & $1$ & $0$ & $4$\\
$3$ & $2$ & $1$ & $0$ & $4$ & $3$\\
$2$ & $1$ & $0$ & $4$ & $3$ & $2$\\
$1$ & $0$ & $4$ & $3$ & $2$ & $1$\\
$0$ & $4$ & $3$ & $2$ & $1$ & $0$\\ 
\end{tabular}
\end{minipage}
\end{table}
\end{exm}

%%%%%%%%%%%%%%%%%%%%%%%%%%%%%%%%%%%%%%

\subsection{When the number of states is a composite number, and the seed is relatively prime to the number of states}
\label{subsec:com1}

In Sections~\ref{subsec:com1} and \ref{subsec:com2}, the case where the number of states of a linear CA is a composite number are considered.
First, in Section~\ref{subsec:com1}, we consider the case where the number of states is composite and the seed is relatively prime to the number of states.
Figure~\ref{fig:sp4} shows spatio-temporal patterns of a $4$-state linear CA, whose transition rule is given by $(T u)_i = u_{i-1} + u_{i+1}$ on $(\mathbb Z/ 4 \mathbb Z)^{\mathbb Z}$, and their seeds are $1$, $2$, and $3$.
In this subsection, we will discuss the cases where the seeds are $1$ and $3$ for this example.
In the case of seed $1$ and $3$, it can be seen that the patterns consisting of positive-valued cells are the same, and observing the details, states $1$ and $3$ are switched.
Figure~\ref{fig:sp6} shows another example, spatio-temporal patterns of a $6$-state linear CA for seed $1$, $2$, $3$, $4$, and $5$ with the same transition rule as in Figure~\ref{fig:sp4}.
The cases where the seed is $1$ and $5$ are discussed in this subsection.
These patterns consisting of positive-valued cells match, and replacements of states other than $0$ has occurred.
The following lemma shows that the spatio-temporal patterns are isomorphic to each other when the number of states is composite and the seed is relatively prime to the number of states, in a manner similar to the lemma for the previous prime number of states.

\begin{figure}[htbp]
\begin{minipage}[b]{0.45\linewidth}
\centering
\includegraphics[width=1.\linewidth]{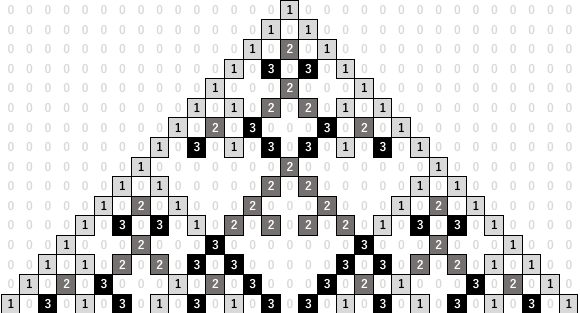}\\
$(a)$ when $a=1$
\end{minipage}
\begin{minipage}[b]{0.05\linewidth}
\quad
\end{minipage}
\begin{minipage}[b]{0.45\linewidth}
\centering
\includegraphics[width=1.\linewidth]{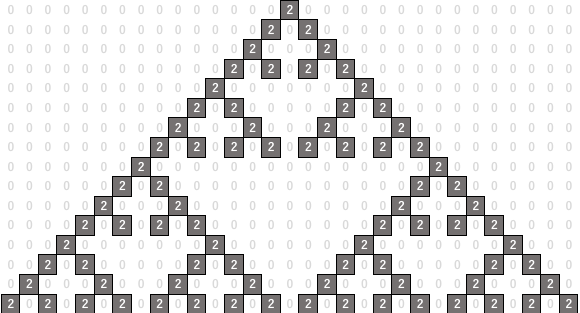}\\
$(b)$ when $a=2$
\end{minipage}\\
\quad \\
\centering
\includegraphics[width=.45\linewidth]{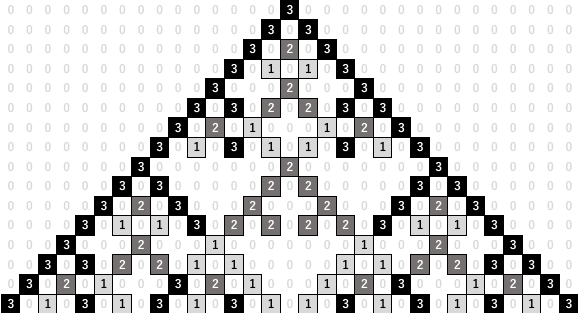}\\
$(c)$ when $a=3$
\caption{Spatio-temporal patterns $\{T^t u_{\langle a \rangle}\}_{t=0}^{15}$ from the initial configuration $u_{\langle a \rangle}$ of a $4$-state linear CA given by $(T u)_i = u_{i-1} + u_{i+1}$ on $(\mathbb Z/ 4 \mathbb Z)^{\mathbb Z}$.}
\label{fig:sp4}
\end{figure}

\begin{figure}[htbp]
\begin{minipage}[b]{0.45\linewidth}
\centering
\includegraphics[width=1.\linewidth]{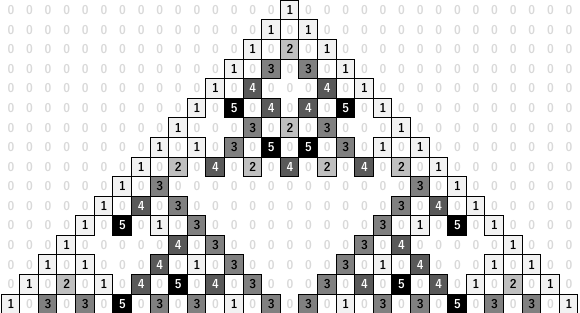}\\
$(a)$ when $a=1$
\end{minipage}
\begin{minipage}[b]{0.05\linewidth}
\quad
\end{minipage}
\begin{minipage}[b]{0.45\linewidth}
\centering
\includegraphics[width=1.\linewidth]{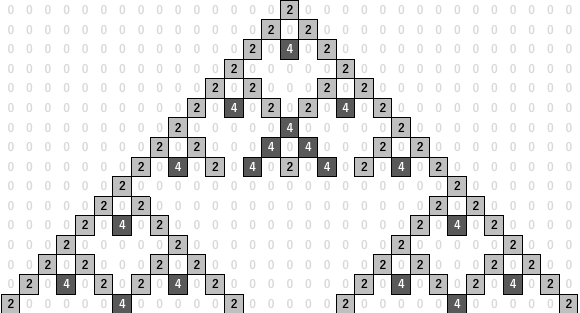}\\
$(b)$ when $a=2$
\end{minipage}\\
\quad \\
\begin{minipage}[b]{0.45\linewidth}
\centering
\includegraphics[width=1.\linewidth]{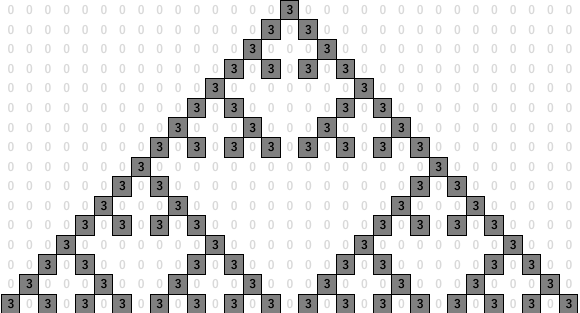}\\
$(c)$ when $a=3$
\end{minipage}
\begin{minipage}[b]{0.05\linewidth}
\quad
\end{minipage}
\begin{minipage}[b]{0.45\linewidth}
\centering
\includegraphics[width=1.\linewidth]{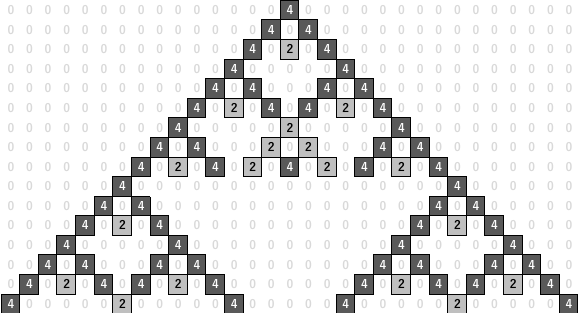}\\
$(d)$ when $a=4$
\end{minipage}\\
\quad \\
\centering
\includegraphics[width=.45\linewidth]{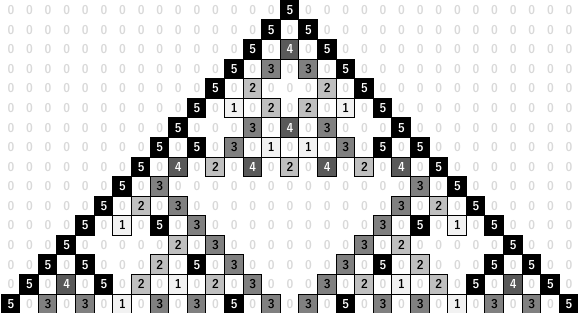}\\
$(e)$ when $a=5$
\caption{Spatio-temporal patterns $\{T^t u_{\langle a \rangle}\}_{t=0}^{15}$ from the initial configuration $u_{\langle a \rangle}$ of a $6$-state linear CA given by $(T u)_i = u_{i-1} + u_{i+1}$ on $(\mathbb Z/ 6 \mathbb Z)^{\mathbb Z}$.}
\label{fig:sp6}
\end{figure}

\begin{lem}
\label{lem:main_com1}
Let $q \in {\mathbb Z}_{> 0}$ be a composite number.
For a $D$-dimensional $q$-state linear CA $(({\mathbb Z} / q {\mathbb Z})^{{\mathbb Z}^D}, T)$, 
\begin{align}
S_T(q, a) \cong S_T(q, \hat{a})
% \{T^t u_{\langle a \rangle}\}_{t=0}^{\infty} \cong \{T^t u_{\langle \hat{a} \rangle}\}_{t=0}^{\infty}
\end{align}
for any seed $a, \hat{a} \in ({\mathbb Z} / q {\mathbb Z})^{\times}$.
\end{lem}

\begin{proof}
This proof method is similar to Lemma~\ref{lem:main_pri}.
Recall that ${\mathbb Z} / q {\mathbb Z}$ is a cyclic group. 
Since $a$ is an element of $({\mathbb Z} / q {\mathbb Z})^{\times}$, $a$ is a generator of ${\mathbb Z} / q {\mathbb Z}$.
Then, all states in ${\mathbb Z} / q {\mathbb Z}$ can appear in the spatio-temporal pattern from the single site seed $u_{\langle a \rangle}$. 

For $k \in ({\mathbb Z} / q {\mathbb Z})^{\times}$, we define a map $f_k$ on ${\mathbb Z} / q {\mathbb Z}$ by $f_k (b) = k b$ for $b \in {\mathbb Z} / q {\mathbb Z}$. 
The map $f_k$ is a homomorphism, since $f_k$ holds $f_k(b+\hat{b}) = f_k(b) + f_k(\hat{b})$ for $b, \hat{b} \in {\mathbb Z} / q {\mathbb Z}$.
Because $k$ and $q$ are relatively prime, we have $\Ker f_k = \{0\}$ and $\im f_k = {\mathbb Z} / q {\mathbb Z}$.
% (Because when $k a = 0$ and $a \neq 0$, we can not take $r \in {\mathbb Z} / q {\mathbb Z}$ such that $k a = r q$ $\pmod q$.)
Hence, $f_k$ is a group automorphism of ${\mathbb Z} / q {\mathbb Z}$. 
% Note that for any $q$, the element $1$ is always included in $({\mathbb Z} / q {\mathbb Z})^{\times}$.
% For each of ${\mathbb Z} / q {\mathbb Z}$ arranged in generation order from generator $1$, $f_k$ gives the corresponding element arranged in generation order from generator $a$.

Thus, for $\hat{a} \in ({\mathbb Z} / q {\mathbb Z})^{\times}$, we take $a \in ({\mathbb Z} / q {\mathbb Z})^{\times}$ such that $(T^t u_{\langle a \rangle})_{\boldsymbol i} = f_k (T^t u_{\langle \hat{a} \rangle})_{\boldsymbol i}$ for any site ${\boldsymbol i} \in {\mathbb Z}^D$ and any time step $t \in {\mathbb Z}_{\geq 0}$.
\end{proof}

We will see the result in Lemma~\ref{lem:main_com1} specifically, by giving the following two examples when the number of states of linear CAs are $4$ and $6$.

\begin{exm}[When $n=4$ and $a \in ({\mathbb Z} / 4 {\mathbb Z})^{\times}$]
\label{exm:lem2_4}

Table~\ref{tab:f_k4} shows the generation orders of ${\mathbb Z} / 4 {\mathbb Z}$ by generator $a \in ({\mathbb Z} / 4 {\mathbb Z})^{\times}$ and the values of $f_a$ for $a \in ({\mathbb Z} / 4 {\mathbb Z})^{\times}$. 
Since $({\mathbb Z} / 4 {\mathbb Z})^{\times} = \{1, 3\}$, the elements are the generators of ${\mathbb Z} / 4 {\mathbb Z}$. 
The order of $b$ at the first row in Table~\ref{tab:f_k4}~$(b)$ is given by the generation order from $1$ as a generator.
We can see that the orders of values appearing in the two tables $(a)$ and $(b)$ correspond. 

\begin{table}[htb]
\caption{Generation orders of the group ${\mathbb Z} / 4 {\mathbb Z}$ by generator $a \in ({\mathbb Z} / 4 {\mathbb Z})^{\times}$ and values of $f_a$ on ${\mathbb Z} / 4 {\mathbb Z}$ for $a \in ({\mathbb Z} / 4 {\mathbb Z})^{\times}$}
\label{tab:f_k4}
\quad \\
\begin{minipage}[c]{0.45\hsize}
\centering
$(a)$ Generation orders of ${\mathbb Z} / 4 {\mathbb Z}$ by $a$\\
\vspace{2mm}
\begin{tabular}{c l}
\hline
$a$ & generation order\\
\hline
$1$ & $1$ $\mapsto$ $2$ $\mapsto$ $3$ $\mapsto$ $0$\\
$3$ & $3$ $\mapsto$ $2$ $\mapsto$ $1$ $\mapsto$ $0$\\
\hline
\end{tabular}
\end{minipage}
\begin{minipage}[c]{0.05\hsize}
\quad
\end{minipage}
\begin{minipage}[c]{0.45\hsize}
\centering
$(b)$ Values of $f_a$ for $b \in {\mathbb Z} / 4 {\mathbb Z}$\\
\vspace{2mm}
\begin{tabular}{c| c c c c}
\hline
$b$ & $1$ & $2$ & $3$ & $0$\\
% \hline
$f_1(b)$ & $1$ & $2$ & $3$ & $0$\\
$f_3(b)$ & $3$ & $2$ & $1$ & $0$\\
\hline
\end{tabular}
\end{minipage}
\end{table}

Because the generators of ${\mathbb Z} / 4 {\mathbb Z}$ are $1$ and $3$, Table~\ref{tab:mod4} shows the additive tables arranged in generation order from the two generators.
Because $f_k$ is a group automorphism, just as when the number of states is prime, the identity element appears at the same sites and positive states are switched.
Though the positive value $2$ appears in the same site, the other elements $1$ and $3$ are switched in these two tables.
Thus, in the spatio-temporal patterns, $0$ corresponds to $0$, and positive-valued states are swapped each other.

\begin{table}[htb]
\caption{Additive tables of ${\mathbb Z} / 4 {\mathbb Z}$ by generators}
\label{tab:mod4}
\quad \\
\begin{minipage}[c]{0.45\hsize}
\centering
$(a)$ Table arranged in the order of generation from generator $1$\\
\vspace{2mm}
\begin{tabular}{c| c c c c}
$+$ & $1$ & $2$ & $3$ & $0$\\
\hline
$1$ & $2$ & $3$ & $0$ & $1$\\
$2$ & $3$ & $0$ & $1$ & $2$\\
$3$ & $0$ & $1$ & $2$ & $3$\\
$0$ & $1$ & $2$ & $3$ & $0$\\ 
\end{tabular}
\end{minipage}
\begin{minipage}[c]{0.05\hsize}
\quad
\end{minipage}
\begin{minipage}[c]{0.45\hsize}
\centering
$(b)$ Table arranged in the order of generation from generator $3$\\
\vspace{2mm}
\begin{tabular}{c| c c c c}
$+$ & $3$ & $2$ & $1$ & $0$\\
\hline
$3$ & $2$ & $1$ & $0$ & $3$\\
$2$ & $1$ & $0$ & $3$ & $2$\\
$1$ & $0$ & $3$ & $2$ & $1$\\
$0$ & $3$ & $2$ & $1$ & $0$\\ 
\end{tabular}
\end{minipage}
\end{table}
\end{exm}

\begin{exm}[When $n=6$ and $a \in ({\mathbb Z} / 6 {\mathbb Z})^{\times}$]
\label{exm:lem2_6}

Since $({\mathbb Z} / 6 {\mathbb Z})^{\times} = \{1, 5\}$, Table~\ref{tab:f_k6} shows the generation orders of ${\mathbb Z} / 6 {\mathbb Z}$ by generator $a \in \{1, 5\}$ and the values of $f_k$ for $k \in \{1, 5\}$. 
The order of $b$ at the first row in Table~\ref{tab:f_k6}~$(b)$ is given by the generation order from $1$.
For a composite number $6$ and $a \in ({\mathbb Z} / 6 {\mathbb Z})^{\times}$, each element is transferred to an element on ${\mathbb Z} / 6 {\mathbb Z}$ by $f_a$.
The generation orders from the two generators $1$ and $5$ correspond to the orders of the values given by $f_a$.

\begin{table}[htb]
\caption{Generation orders of the group ${\mathbb Z} / 6 {\mathbb Z}$ by generator $a \in ({\mathbb Z} / 6 {\mathbb Z})^{\times}$ and values of $f_a$ on ${\mathbb Z} / 6 {\mathbb Z}$ for $a \in ({\mathbb Z} / 6 {\mathbb Z})^{\times}$}
\label{tab:f_k6}
\quad \\
\begin{minipage}[c]{0.55\hsize}
\centering
$(a)$ Generation orders of ${\mathbb Z} / 6 {\mathbb Z}$ by $a$\\
\vspace{2mm}
\begin{tabular}{c l}
\hline
$a$ & generation order\\
\hline
$1$ & $1$ $\mapsto$ $2$ $\mapsto$ $3$ $\mapsto$ $4$ $\mapsto$ $5$ $\mapsto$ $0$\\
$5$ & $5$ $\mapsto$ $4$ $\mapsto$ $3$ $\mapsto$ $2$ $\mapsto$ $1$ $\mapsto$ $0$\\
\hline
\end{tabular}
\end{minipage}
\begin{minipage}[c]{0.05\hsize}
\quad
\end{minipage}
\begin{minipage}[c]{0.35\hsize}
\centering
$(b)$ Values of $f_a$ for $b \in {\mathbb Z} / 6 {\mathbb Z}$\\
\vspace{2mm}
\begin{tabular}{c| c c c c c c}
\hline
$b$ & $1$ & $2$ & $3$ & $4$ & $5$ & $0$\\
% \hline
$f_1(b)$ & $1$ & $2$ & $3$ & $4$ & $5$ & $0$\\
$f_5(b)$ & $5$ & $4$ & $3$ & $2$ & $1$ & $0$\\
\hline
\end{tabular}
\end{minipage}
\end{table}

For the generators $1$ and $5$, Table~\ref{tab:mod6} shows the additive tables of ${\mathbb Z} / 6 {\mathbb Z}$ where elements are ordered by the generation orders.
In these tables, the identity element appears in the same sites and positive elements are switched in the two tables. 
Because by a group automorphism $f_a$, states $1$ and $5$ are switched and states $2$ and $4$ are switched, for the spatio-temporal patterns of a $6$-state linear CA with the seed $a \in ({\mathbb Z} / 6 {\mathbb Z})^{\times}$, when the seed is changed, the identity element $0$ appears in the same sites, but the positive-valued states are switched.

\begin{table}[htb]
\caption{Additive tables of ${\mathbb Z} / 6 {\mathbb Z}$ by generators}
\label{tab:mod6}
\quad \\
\begin{minipage}[c]{0.45\hsize}
\centering
$(a)$ Table arranged in the order of generation from generator $1$\\
\vspace{2mm}
\begin{tabular}{c| c c c c c c}
$+$ & $1$ & $2$ & $3$ & $4$ & $5$ & $0$\\
\hline
$1$ & $2$ & $3$ & $4$ & $5$ & $0$ & $1$\\
$2$ & $3$ & $4$ & $5$ & $0$ & $1$ & $2$\\
$3$ & $4$ & $5$ & $0$ & $1$ & $2$ & $3$\\
$4$ & $5$ & $0$ & $1$ & $2$ & $3$ & $4$\\
$5$ & $0$ & $1$ & $2$ & $3$ & $4$ & $5$\\
$0$ & $1$ & $2$ & $3$ & $4$ & $5$ & $0$\\ 
\end{tabular}
\end{minipage}
\begin{minipage}[c]{0.05\hsize}
\quad
\end{minipage}
\begin{minipage}[c]{0.45\hsize}
\centering
$(b)$ Table arranged in the order of generation from generator $5$\\
\vspace{2mm}
\begin{tabular}{c| c c c c c c}
$+$ & $5$ & $4$ & $3$ & $2$ & $1$ & $0$\\
\hline
$5$ & $4$ & $3$ & $2$ & $1$ & $0$ & $5$\\
$4$ & $3$ & $2$ & $1$ & $0$ & $5$ & $4$\\
$3$ & $2$ & $1$ & $0$ & $5$ & $4$ & $3$\\
$2$ & $1$ & $0$ & $5$ & $4$ & $3$ & $2$\\
$1$ & $0$ & $5$ & $4$ & $3$ & $2$ & $1$\\
$0$ & $5$ & $4$ & $3$ & $2$ & $1$ & $0$\\ 
\end{tabular}
\end{minipage}
\end{table}
\end{exm}

%%%%%%%%%%%%%%%%%%%%%%%%%%%%%%%%%%%%%%

\subsection{When the number of states is a composite number, and the seed is not relatively prime to the number of states}
\label{subsec:com2}

In this subsection, we consider the case where the number of states is composite and was not considered in the previous subsection, Section~\ref{subsec:com1}.
In other words, we consider the case where the number of states $n$ of a linear CA is composite and the seed $a$ is not relatively prime to $n$.

Here, we see Figure~\ref{fig:sp4} again, that shows the spatio-temporal patterns of a $4$-state linear CA.
In Lemma~\ref{lem:main_com1}, we have already shown that the spatio-temporal patterns of seeded $1$ and $3$ are isomorphic. 
In this subsection, we give results for the spatio-temporal pattern in the remaining case of seed $2$. 
In this case, we show that the spatio-temporal pattern is isomorphic to the pattern of the elementary CA Rule $90$ with the number of states $2$, although it is different from the spatio-temporal patterns from other seeds.
Furthermore, we see the spatio-temporal patterns of the $6$-state linear CA in Figure~\ref{fig:sp6} again.
In Lemma~\ref{lem:main_com1}, we have already shown that the spatio-temporal patterns for seed $1$ and $5$ are isomorphic, so here we consider the remaining patterns for seed $2$, $3$ and $4$.
Observing these patterns, we can notice that the patterns of the positive-valued states for seed $2$ and $4$ are the same, but in the case of seed $3$, it is different from them.
In the following lemma, we show that the spatio-temporal patterns for seed $2$ and $4$ are isomorphic to a spatio-temporal pattern of a $3$-state linear CA, while for the seed $3$ the pattern is isomorphic to the pattern of a $2$-state linear CA. 
More generally, we show that for a composite number $n$ and the seed $a$ that is not relatively prime to $n$, the spatio-temporal pattern for seed $a$ with $n$ states is isomorphic to the spatio-temporal pattern of a linear CA for seed $a/\gcd(n, a)$ with $n/\gcd(n, a)$ states.

\begin{lem}
\label{lem:main_com2}
Let $q$ be a composite number. 
For $a \in ({\mathbb Z} / q {\mathbb Z}) \backslash \{ ({\mathbb Z} / q {\mathbb Z})^{\times} \cup \{0\}\}$, let $r = q / \gcd(q, a)$. %, where $\gcd (q, a)$ denotes the greatest common divisor of $q$ and $a$. 
We consider two CAs; a $D$-dimensional $q$-state linear CA $(({\mathbb Z} / q {\mathbb Z})^{{\mathbb Z}^D}, T)$, and a $D$-dimensional $r$-state linear CA $(({\mathbb Z} / r {\mathbb Z})^{{\mathbb Z}^D}, T)$.
Note that the CAs have the same transition rule denoted by $T$. 
Then, 
\begin{align}
S_T(q, a) \cong S_T(r, a/ \gcd(q, a)).
% \{T_{(q)}^t u_{\langle a \rangle}\}_{t=0}^{\infty} \cong \{T_{(r)}^t u_{\langle a/ \gcd(q, a) \rangle}\}_{t=0}^{\infty}.
\end{align}
%\begin{align}
%(T^t u_{\langle a \rangle})_{\boldsymbol i} = \gcd(q, a) (\hat{T}^t u_{\langle a/\gcd(q, a) \rangle})_{\boldsymbol i}.
%\end{align}
\end{lem}

\begin{proof}
For the group ${\mathbb Z} / q {\mathbb Z}$, it is known that for each divisor $r$ of $q$, there exists one subgroup whose order is $r$.
Let $a \in ({\mathbb Z} / q {\mathbb Z}) \backslash \{ ({\mathbb Z} / q {\mathbb Z})^{\times} \cup \{0\}\}$, and let $r = q / \gcd(q, a)$. 
Then, $1< r < q$, and there exists a proper subgroup of ${\mathbb Z} / q {\mathbb Z}$, $\gcd (q, a) {\mathbb Z} / q {\mathbb Z}$ with the order $r$.
Although $a$ is not a generator of ${\mathbb Z} / q {\mathbb Z}$, $a$ is a generator of the proper subgroup $\gcd (q, a) {\mathbb Z} / q {\mathbb Z}$.
For $r = q / \gcd(q, k)$ with $k \in ({\mathbb Z} / q {\mathbb Z}) \backslash \{ ({\mathbb Z} / q {\mathbb Z})^{\times} \cup \{0\}\}$, given the map $f_k : \gcd (q, k) {\mathbb Z} / q {\mathbb Z} \to {\mathbb Z} / r {\mathbb Z}$ with $f_k(b) = b/ \gcd(q, k)$ for $b \in \gcd (q, k) {\mathbb Z} / q {\mathbb Z}$, it is an isomorphism.
\end{proof}

For Lemma~\ref{lem:main_com2}, we will give the following two examples when the number of states is $4$ and $6$.

\begin{exm}[When $n=4$ and $a \in ({\mathbb Z} / 4 {\mathbb Z}) \backslash \{({\mathbb Z} / 4 {\mathbb Z})^{\times} \cup \{0\} \}$]
\label{exm:lem3_4}

First, $({\mathbb Z} / 4 {\mathbb Z}) \backslash \{({\mathbb Z} / 4 {\mathbb Z})^{\times} \cup \{0\} \} = \{2\}$.
When $n=4$ and $a=2$, $\gcd(n, a) = 2$, and then we define a map $f_2 : 2 {\mathbb Z} / 4 {\mathbb Z} \to {\mathbb Z} / 2 {\mathbb Z}$ by $b/2$ for $b \in 2 {\mathbb Z} / 4 {\mathbb Z}$.
Table~\ref{tab:f_k4_com2}~$(a)$ shows the generation order of a subgroup of ${\mathbb Z} / 4 {\mathbb Z}$, $2 {\mathbb Z} / 4 {\mathbb Z}$, and Table~\ref{tab:f_k4_com2}~$(b)$ shows the values of $f_2$. 
It can be seen that the states $2$ and $0$ in $2 {\mathbb Z} / 4 {\mathbb Z}$ correspond to the states $1$ and $0$ in ${\mathbb Z} / 2 {\mathbb Z}$, respectively.

\begin{table}[htb]
\caption{Generation order of the proper subgroup of ${\mathbb Z} / 4 {\mathbb Z}$ by $a \in ({\mathbb Z} / 4 {\mathbb Z}) \backslash \{({\mathbb Z} / 4 {\mathbb Z})^{\times} \cup \{0\} \}$ and values of $f_a$ from $2 {\mathbb Z} / 4 {\mathbb Z}$ to ${\mathbb Z} / 2 {\mathbb Z}$ for $a \in ({\mathbb Z} / 4 {\mathbb Z}) \backslash \{({\mathbb Z} / 4 {\mathbb Z})^{\times} \cup \{0\} \}$}
\label{tab:f_k4_com2}
\quad \\
\begin{minipage}[c]{0.45\hsize}
\centering
$(a)$ Generation order of $2 {\mathbb Z} / 4 {\mathbb Z}$ by $a$\\
\vspace{2mm}
\begin{tabular}{c l}
\hline
$a$ & generation order\\
\hline
$2$ & $2$ $\mapsto$ $0$ \\
\hline
\end{tabular}
\end{minipage}
\begin{minipage}[c]{0.05\hsize}
\quad
\end{minipage}
\begin{minipage}[c]{0.45\hsize}
\centering
$(b)$ Values of $f_a$ for $b \in 2 {\mathbb Z} / 4 {\mathbb Z}$\\
\vspace{2mm}
\begin{tabular}{c| c c}
\hline
$b$ & $2$ & $0$\\
% \hline
$f_2(b)$ & $1$ & $0$\\
\hline
\end{tabular}
\end{minipage}
\end{table}

Table~\ref{tab:mod4_2} shows the additive table of $2 {\mathbb Z} / 4 {\mathbb Z}$, and Table~\ref{tab:mod2_1} shows the additive table of ${\mathbb Z} / 2 {\mathbb Z}$ where elements are ordered by the generation orders.
Comparing them, we can see that the state $2$ in Table~\ref{tab:mod4_2} corresponds to $1$ in Table~\ref{tab:mod2_1}, and $0$ remains $0$.
From these additive tables, we see that the spatio-temporal patterns of the two linear CAs are isomorphic. 
Because the map $f_2$ is an isomorphism, the states $2$ and $0$ in the spatio-temporal pattern of a $4$-state linear CA for the seed $2$ correspond to the states $1$ and $0$ in the spatio-temporal pattern of a $2$-state linear CA for the seed $1$, respectively.

\begin{table}[htbp]
\begin{minipage}[c]{0.45\hsize}
\caption{Additive table of $2 {\mathbb Z} / 4 {\mathbb Z}$ (arranged in the order of generation from the element $2$)}
\label{tab:mod4_2}
\quad \\
%\begin{minipage}[c]{0.31\hsize}
\centering
%$(a)$ Table arranged in the order of generation from the element $2$\\
%\vspace{2mm}
\begin{tabular}{c| c c}
$+$ & $2$ & $0$ \\
\hline
$2$ & $0$ & $2$\\
$0$ & $2$ & $0$\\
\end{tabular}
%\end{minipage}
\end{minipage}
\begin{minipage}[c]{0.05\hsize}
\quad
\end{minipage}
\begin{minipage}[c]{0.45\hsize}
\caption{Additive table of ${\mathbb Z} / 2 {\mathbb Z}$ (arranged in the order of generation from the element $1$)}
\label{tab:mod2_1}
\quad \\
%\begin{minipage}[c]{0.31\hsize}
\centering
%$(a)$ Table arranged in the order of generation from the element $2$\\
%\vspace{2mm}
\begin{tabular}{c| c c}
$+$ & $1$ & $0$ \\
\hline
$1$ & $0$ & $1$\\
$0$ & $1$ & $0$\\
\end{tabular}
%\end{minipage}
\end{minipage}
\end{table}
\end{exm}

\begin{exm}[When $n=6$ and $a \in ({\mathbb Z} / 6 {\mathbb Z}) \backslash \{({\mathbb Z} / 6 {\mathbb Z})^{\times} \cup \{0\} \}$]
\label{exm:lem3_6}

We have $({\mathbb Z} / 6 {\mathbb Z}) \backslash \{({\mathbb Z} / 6 {\mathbb Z})^{\times} \cup \{0\} \} = \{2, 3, 4\}$.
We define a map $f_a : a {\mathbb Z} / 6 {\mathbb Z} \to {\mathbb Z} / (6/a) {\mathbb Z}$ by $b/a$ for $b \in a {\mathbb Z} / 6 {\mathbb Z}$.
Because $\gcd(6, 2) = \gcd(6, 4) = 2$ and $\gcd(6, 3)= 3$, we separate the cases where the seed is $2$ and $4$ from the case where the seed is $3$.

\begin{table}[htb]
\caption{Generation orders of a subgroup $2 {\mathbb Z} / 6 {\mathbb Z}$, and values of $f_a$ for $a \in \{2, 4\}$ from $2 {\mathbb Z} / 6 {\mathbb Z}$ to ${\mathbb Z} / 3 {\mathbb Z}$}
\label{tab:f_k6_com2}
\quad \\
\begin{minipage}[c]{0.35\hsize}
\centering
$(a)$ Generation orders of $2 {\mathbb Z} / 6 {\mathbb Z}$\\
\vspace{2mm}
\begin{tabular}{c l}
\hline
$a$ & generation order\\
\hline
$2$ & $2$ $\mapsto$ $4$ $\mapsto$ $0$\\
$4$ & $4$ $\mapsto$ $2$ $\mapsto$ $0$\\
\hline
\end{tabular}
\end{minipage}
\begin{minipage}[c]{0.05\hsize}
\quad
\end{minipage}
\begin{minipage}[c]{0.55\hsize}
\centering
$(b)$ Values of $f_a$ for $b \in 2 {\mathbb Z} / 6 {\mathbb Z}$\\
\vspace{2mm}
\begin{minipage}[c]{0.45\hsize}
\vspace{2mm}
\begin{tabular}{c| c c c}
\hline
$b$ & $2$ & $4$ & $0$\\
% \hline
$f_2(b)$ & $1$ & $2$ & $0$ \\
\hline
\end{tabular}
\end{minipage}
\begin{minipage}[c]{0.05\hsize}
\quad 
\end{minipage}
\begin{minipage}[c]{0.45\hsize}
\vspace{2mm}
\begin{tabular}{c| c c c}
\hline
$b$ & $4$ & $2$ & $0$\\
$f_4(b)$ & $2$ & $1$ & $0$ \\
\hline
\end{tabular}
\end{minipage}
\end{minipage}
\end{table}

First, when the seed $a$ is $2$ or $4$, we consider a subgroup of ${\mathbb Z} / 6 {\mathbb Z}$, $2 {\mathbb Z} / 6 {\mathbb Z}$, because $\gcd(6, a) = 2$. 
Table~\ref{tab:f_k6_com2}~$(a)$ shows the generation orders of $2 {\mathbb Z} / 6 {\mathbb Z}$ from generators, and Table~\ref{tab:f_k6_com2}~$(b)$ shows the values of $f_a$ for $a \in \{2, 4\}$. 
The isomorphism $f_a$ allows us to reduce the number of states from $6$ to $3$.
Tables~\ref{tab:mod6_2}~$(a)$ and $(c)$ show additive tables of $2 {\mathbb Z} / 6 {\mathbb Z}$ arranged in generation order from generators, $2$ and $4$.
Table~\ref{tab:mod6_2}~$(a)$ starting from the seed $2$ is transferred by $f_2$ to the $3$-state additive table in Table~\ref{tab:mod3}~$(a)$, and Table~\ref{tab:mod6_2}~$(c)$ starting from the seed $4$ is transferred by $f_2$ to the $3$-state additive table in Table~\ref{tab:mod3}~$(b)$.
Observing these tables, it can be seen that the states $2$ and $4$ in $2 {\mathbb Z} / 6 {\mathbb Z}$ are transferred to the states $1$ and $2$ in ${\mathbb Z} / 3 {\mathbb Z}$, respectively.

\begin{table}[htb]
\caption{Additive tables of subgroups of ${\mathbb Z} / 6 {\mathbb Z}$}
\label{tab:mod6_2}
\quad \\
\begin{minipage}[c]{0.3\hsize}
\centering
$(a)$ Additive table of $2 {\mathbb Z} / 6 {\mathbb Z}$ (arranged in the order of generation from $2$)\\
\vspace{2mm}
\begin{tabular}{c| c c c}
$+$ & $2$ & $4$ & $0$ \\
\hline
$2$ & $4$ & $0$ & $2$\\
$4$ & $0$ & $2$ & $4$\\
$0$ & $2$ & $4$ & $0$\\
\end{tabular}
\end{minipage}
\begin{minipage}[c]{0.03\hsize}
\quad
\end{minipage}
\begin{minipage}[c]{0.3\hsize}
\centering
$(b)$ Additive table of $3 {\mathbb Z} / 6 {\mathbb Z}$ (arranged in the order of generation from $3$)\\
\vspace{2mm}
\begin{tabular}{c| c c}
$+$ & $3$ & $0$ \\
\hline
$3$ & $0$ & $3$ \\
$0$ & $3$ & $0$ \\
\end{tabular}
\end{minipage}
\begin{minipage}[c]{0.03\hsize}
\quad
\end{minipage}
\begin{minipage}[c]{0.3\hsize}
\centering
$(c)$Additive table of $2 {\mathbb Z} / 6 {\mathbb Z}$ (arranged in the order of generation from $4$)\\
\vspace{2mm}
\begin{tabular}{c| c c c}
$+$ & $4$ & $2$ & $0$ \\
\hline
$4$ & $2$ & $0$ & $4$ \\
$2$ & $0$ & $4$ & $2$ \\
$0$ & $4$ & $2$ & $0$ \\
\end{tabular}
\end{minipage}
\end{table}

\begin{table}[htb]
\caption{Generation order of a subgroup $3 {\mathbb Z} / 6 {\mathbb Z}$, and values of $f_a$ for $a = 3$ from $3 {\mathbb Z} / 6 {\mathbb Z}$ to ${\mathbb Z} / 2 {\mathbb Z}$}
\label{tab:f_k6_com3}
\quad \\
\begin{minipage}[c]{0.45\hsize}
\centering
$(a)$ Generation order of $3 {\mathbb Z} / 6 {\mathbb Z}$\\
\vspace{2mm}
\begin{tabular}{c l}
\hline
$a$ & generation order\\
\hline
$3$ & $3$ $\mapsto$ $0$\\
\hline
\end{tabular}
\end{minipage}
\begin{minipage}[c]{0.05\hsize}
\quad
\end{minipage}
\begin{minipage}[c]{0.45\hsize}
\centering
$(b)$ Values of $f_a$ for $b \in 3 {\mathbb Z} / 6 {\mathbb Z}$\\
\vspace{2mm}
\begin{tabular}{c| c c}
\hline
$b$ & $3$ & $0$ \\
% \hline
$f_3(b)$ & $1$ & $0$ \\
\hline
\end{tabular}
\end{minipage}
\end{table}

Next, we study the spatio-temporal pattern of a $6$-state linear CA for the seed $3$.
Since $\gcd(6, 3) = 3$, we consider a subgroup of ${\mathbb Z} / 6 {\mathbb Z}$, $3 {\mathbb Z} / 6 {\mathbb Z}$.
Table~\ref{tab:f_k6_com3}~$(a)$ shows the generation order of a subgroup $3 {\mathbb Z} / 6 {\mathbb Z}$, and Table~\ref{tab:f_k6_com3}~$(b)$ shows the values of $f_a$ from $3 {\mathbb Z} / 6 {\mathbb Z}$ to ${\mathbb Z} / 2 {\mathbb Z}$ for $a = 3$. 
Table~\ref{tab:mod6_2}~$(b)$ shows the additive table of $3 {\mathbb Z} / 6 {\mathbb Z}$ arranged in order of generation from the generator $3$.
By $f_3$, this table can be transferred to the additive table ${\mathbb Z} / 2 {\mathbb Z}$ in Table~\ref{tab:mod2_1}.
Thus, we see that the spatio-temporal pattern of a $6$-state linear CA for the seed $3$ is isomorphic to the spatio-temporal pattern of a $2$-state linear CA for the seed $1$ with the same transition rule.
\end{exm}

%%%%%%%%%%%%%%%%%%%%%%%%%%%%%%%%%%%%%%

\subsection{Summary of the three previous cases (when the number of states is any integer greater than or equal to $2$)}
\label{subsec:all}

Summarizing the previous three sections, Sections~\ref{subsec:pri}, \ref{subsec:com1}, and \ref{subsec:com2}, we give the following theorem.

\begin{thm}
\label{thm:main}
Let $n$ be an integer greater than or equal to $2$.
For $a \in ({\mathbb Z} / n {\mathbb Z}) \backslash \{0\}$, let $r = n / \gcd(n, a)$. 
Let $(({\mathbb Z} / n {\mathbb Z})^{{\mathbb Z}^D}, T)$ be a $D$-dimensional $n$-state linear CA, and let $(({\mathbb Z} / r {\mathbb Z})^{{\mathbb Z}^D}, T)$ be a $D$-dimensional $r$-state linear CA whose transition rule is the same as $(({\mathbb Z} / n {\mathbb Z})^{{\mathbb Z}^D}, T)$. 
The spatio-temporal patterns of the two CAs hold the following relationship:
\begin{align}
S_T(n, a) \cong S_T(r, 1). 
% \{T_{(n)}^t u_{\langle a \rangle}\}_{t=0}^{\infty} \cong \{T_{(r)}^t u_{\langle 1 \rangle}\}_{t=0}^{\infty}.
\end{align}
\end{thm}

\begin{proof}
By Lemmas~\ref{lem:main_pri} and \ref{lem:main_com1}, for $n \in {\mathbb Z}_{\geq 2}$ and $a \in ({\mathbb Z} / n {\mathbb Z})^{\times}$,  $S_T(n, a) \cong S_T(n, 1)$, because $1 \in ({\mathbb Z} / n {\mathbb Z})^{\times}$ for any $n$.
% By Lemmas~\ref{lem:main_pri} and \ref{lem:main_com1}, if $n$ is prime, or if $n$ is a composite number and $\gcd(n, a)=1$ for $a \in ({\mathbb Z} / n {\mathbb Z}) \backslash \{0\}$,  $S_T(n, a) \cong S_T(n, 1)$.
% $\{T^t u_{\langle a \rangle}\}_{t=0}^{\infty} \cong \{T^t u_{\langle 1 \rangle}\}_{t=0}^{\infty}$. 
By Lemma~\ref{lem:main_com2}, $S_T(n, a) \cong S_T(r, a/ \gcd(n, a))$, and by Lemmas~\ref{lem:main_pri} or \ref{lem:main_com1}, $S_T(n, a) \cong S_T(r, 1)$, because $\gcd( n/\gcd(n, a), a/\gcd(n, a))=1$.
\end{proof}

Therefore, it is sufficient to consider only the spatio-temporal pattern of a linear CA for the case of the seed $1$ as representative, without considering the spatio-temporal patterns for all seeds.

%%%%%%%%%%%%%%%%%%%%%%

\section{Concluding Remarks}
\label{sec:con}

In this paper, we proved that it is sufficient to consider the initial configuration of linear CAs only when the seed is $1$.
In Section~\ref{subsec:pri}, we discussed that when the number of states is prime, the spatio-temporal patterns of a linear CA from the single site seed are isomorphic regardless of the seed.
Next, in the case where the number of states is composite, we proved the following two cases.
In Section~\ref{subsec:com1}, we showed that when the number of states is composite and the seed is relatively prime to the number of states, their spatio-temporal patterns are isomorphic to each other even if the seed is switched.
In Section~\ref{subsec:com2}, we showed that when the number of states is composite and the seed is not relatively prime to the number of states, there exist spatio-temporal patterns with a smaller number of states that are isomorphic to the spatio-temporal pattern.
In Section~\ref{subsec:all}, by integrating these three cases, we were able to show that for any spatio-temporal pattern from any seed with any number of states, there is always a spatio-temporal pattern from the seed $1$ isomorphic to it.

The authors have recently studied to represent and classify fractals generated from CAs by one-variable functions \cite{kawa2022, kawa2022p}.
Since the present results in this paper show that it is sufficient to consider only the case where the seed is $1$ for the orbits of linear CAs with any number of states, in future work, we plan to study one-variable functions representing fractals generated from CAs whose possible states are more than or equal to $3$.
We would like to find out what functions (e.g. pathological functions) emerge and use them to classify fractals. 
Furthermore, we would like to consider applying the results of this study to nonlinear CAs.
In the case of nonlinear CAs, due to the definition of transition rules, it is not possible to apply the present proof using the group theory for linear cases.
However, we would like to consider applying the partial results by attributing them to the linear case.
%%%%%%%%%%%%%%%%%%%%%%%%%%%%%%%%%%%%%%

\section*{Acknowledgment}
This work was partly supported by a Grant-in-Aid for Scientific Research (22K03435) funded by the Japan Society for the Promotion of Science.

\section*{Data Availability Statement}
The data that supports the findings of this work are available within this paper.

% \bibliographystyle{unsrt}
% \bibliography{papers}

\begin{thebibliography}{10}

\bibitem{hedlund1969}
G.~A. Hedlund.
\newblock Endomorphisms and automorphisms of the shift dynamical system.
\newblock {\em Mathematical systems theory}, 3:320--375, 1969.

\bibitem{blanchard1997}
Fran\c{c}ois Blanchard, Petr K\ocirc{u}rka, and Alejandro Maass.
\newblock Topological and measure-theoretic properties of one-dimensional
  cellular automata.
\newblock {\em Physica D: Nonlinear Phenomena}, 103(1-4):86--99, 1997.
\newblock Lattice dynamics (Paris, 1995).

\bibitem{kurka2001}
Petr K\ocirc{u}rka.
\newblock Topological dynamics of cellular automata.
\newblock In {\em Codes, systems, and graphical models ({M}inneapolis, {MN},
  1999)}, volume 123 of {\em The IMA Volumes in Mathematics and its Applications}, 447--485. Springer,
  New York, 2001.

\bibitem{amoroso1971}
Serafino Amoroso and Gerald Cooper.
\newblock Tessellation structures for reproduction of arbitrary patterns.
\newblock {\em Journal of Computer and System Sciences}, 5(5):455--464, 1971.

\bibitem{ostrand1971}
Thomas~J. Ostrand.
\newblock Pattern reproduction in tessellation automata of arbitrary dimension.
\newblock {\em Journal of Computer and System Sciences}, 5:623--628, 1971.

\bibitem{willson1984}
Stephen~J. Willson.
\newblock Cellular automata can generate fractals.
\newblock {\em Discrete Applied Mathematics}, 8(1):91--99, 1984.

\bibitem{coombes2009}
Stephen Coombes.
\newblock The geometry and pigmentation of seashells.
\newblock {\em Nottingham: Department of Mathematical Sciences, University of
  Nottingham}, 2009.

\bibitem{gerhardt1989}
M.~Gerhardt and H.~Schuster.
\newblock A cellular automaton describing the formation of spatially ordered
  structures in chemical systems.
\newblock {\em Physica D: Nonlinear Phenomena}, 36(3):209--221, 1989.

\bibitem{fuks1997}
Henryk Fuk\'{s}.
\newblock Solution of the density classification problem with two cellular
  automata rules.
\newblock {\em Physical Review E}, 55:R2081--R2084, Mar 1997.

\bibitem{matsumoto1998}
Makoto Matsumoto.
\newblock Simple cellular automata as pseudorandom m-sequence generators for
  built-in self-test.
\newblock {\em ACM Transactions on Modeling and Computer Simulation (TOMACS)},
  8(1):31--42, 1998.

\bibitem{takahashi1992}
Satoshi Takahashi.
\newblock Self-similarity of linear cellular automata.
\newblock {\em Journal of Computer and System Sciences}, 44:114--140, 1992.

\bibitem{haeseler1993}
F.~von Haeseler, H.-O. Peitgen, and G.~Skordev.
\newblock Cellular automata, matrix substitutions and fractals.
\newblock {\em Ann. Math. Artificial Intelligence}, 8(3-4):345--362, 1993.
\newblock Theorem proving and logic programming (1992).

\bibitem{wolfram2002}
Stephen Wolfram.
\newblock {\em A New Kind of Science}.
\newblock Wolfram Media, 2002.

\bibitem{kawa2022}
Akane Kawaharada.
\newblock Singular function emerging from one-dimensional elementary cellular
  automaton rule 150.
\newblock {\em Discrete and Continuous Dynamical Systems - Series B},
  27(4):2115--2128, 2022.

\bibitem{kawa2022p}
Akane Kawaharada.
\newblock Cellular automata that generate symmetrical patterns give singular
  functions.
\newblock {\em Physica D: Nonlinear Phenomena}, 439:133428, 2022.

\end{thebibliography}

\end{document}